\documentclass[review,12pt]{elsarticle}

\usepackage{amssymb,amsmath,amsthm,graphicx}
\usepackage{enumerate}
\usepackage{float}
\usepackage{stmaryrd,mathbbol,mathrsfs} 
\usepackage{multicol}
\usepackage{float}
\usepackage{stackrel}
\usepackage{color}
\usepackage{tikz}
\usetikzlibrary{matrix,arrows}
\usepackage{capt-of}
\usepackage{array}

\newtheorem{theorem}{Theorem}[section]
\newtheorem*{theorem*}{Theorem}
\newtheorem{lemma}[theorem]{Lemma}

\newtheorem{corollary}[theorem]{Corollary}

\theoremstyle{remark}

\theoremstyle{definition}

\newtheorem{definition}[theorem]{Definition}

\makeatletter
\newcommand{\Spvek}[2][r]{%
  \gdef\@VORNE{1}
  \left(\hskip-\arraycolsep%
    \begin{array}{#1}\vekSp@lten{#2}\end{array}%
  \hskip-\arraycolsep\right)}

\def\vekSp@lten#1{\xvekSp@lten#1;vekL@stLine;}
\def\vekL@stLine{vekL@stLine}
\def\xvekSp@lten#1;{\def\temp{#1}%
  \ifx\temp\vekL@stLine
  \else
    \ifnum\@VORNE=1\gdef\@VORNE{0}
    \else\@arraycr\fi%
    #1%
    \expandafter\xvekSp@lten
  \fi}
\makeatother
\journal{}

\makeatletter
\def\ps@pprintTitle{%
 \let\@oddhead\@empty
 \let\@evenhead\@empty
 \def\@oddfoot{}%
 \let\@evenfoot\@oddfoot}
\makeatother

\begin{document}

\begin{frontmatter}



\title{Congruences and the discrete Sugeno integrals on bounded distributive lattices \footnote{Preprint of an article published by Elsevier in the Information Sciences 367-368 (2016), 443-448. It is available online at: \newline www.sciencedirect.com/science/article/pii/S002002551630439X}}


\author[up]{Radom\'ir Hala\v{s}}
\ead{radomir.halas@upol.cz}

\author[stu,ost]{Radko Mesiar}
\ead{radko.mesiar@stuba.sk}

\author[up,sav]{Jozef P\'ocs}
\ead{pocs@saske.sk}

\address[up]{Department of Algebra and Geometry, Faculty of Science, Palack\'y University Olomouc, 17. listopadu 12, 771 46 Olomouc, Czech Republic}
\address[stu]{Department of Mathematics and Descriptive Geometry, Faculty of Civil Engineering, Slovak University of Technology in Bratislava, Radlinsk\'eho 11, 810 05 Bratislava 1, Slovakia}
\address[ost]{University of Ostrava, Institute for Research and Applications of Fuzzy Modeling, NSC Centre of Excellence IT4Innovations, 30. dubna 22, 701 03 Ostrava, Czech Republic}
\address[sav]{Mathematical Institute, Slovak Academy of Sciences, Gre\v s\'akova 6, 040 01 Ko\v sice, Slovakia}
\begin{abstract}
We study compatible aggregation functions on a general bounded
distributive lattice $L$, where the compatibility is related to the congruences
on $L$. As a by-product, a new proof of an earlier result of G. Gr\"atzer is obtained. Moreover, 
our results yield a new characterization of discrete Sugeno integrals on bounded distributive lattices.

\end{abstract}


\end{frontmatter}

\section{Introduction}

Several modern approaches in the area of information sciences do not deal with numerical information, but more general types of data are considered, in particular data which are elements of bounded lattices. 
As a typical example, recall
Goguen's introduction of $L$-fuzzy sets \cite{Go} and several of related particular
concepts based on particular bounded lattices, such as fuzzy type-$2$ sets
proposed by Zadeh \cite{Za} or Atanassov's intuitionistic fuzzy sets \cite{At}. Similarly, bounded lattices are exploited to represent ordinal information dealing with linguistic scales \cite{Za1} etc.
This is especially important in domains, where the essential information is either not available, or superfluous, and only the ordinal relationships are of interest. We can mention several recent papers devoted to a deeper study of lattice-based data see, e.g., \cite{N2,N1}.
Aggregation on bounded lattices belongs to basic tools of lattice-based data, and thus a deeper development of aggregation on lattices is an important and hot topic. To illustrate this fact and to give the reader more broader overview on recent results on aggregation on lattices, we recommend the papers \cite{N4,N3,N5}, 
or recent papers \cite{HaPo1,HaPo2}, where aggregation functions are studied by means of a clone theory.

Recall that an $n$-ary aggregation function $g\colon L^n\to L$, where $(L,0,1,\leq)$ is a bounded lattice (or, more generally, a bounded poset) is characterized by its non-decreasingness in each coordinate and by two boundary conditions $g(0,\dots,0)=0$ and $g(1,\dots,1)=1$. Typical aggregation functions on a bounded distributive lattice are lattice polynomials $p\colon L^n\to L$ given by 

$$ p(a_1,\dots,a_n)=\bigvee_{I\in \mathcal{J}}\big(\bigwedge_{i\in I}a_i\big),$$
where $\mathcal{J}\subseteq 2^{\{1,\dots,n\}}$ is a non-empty subset of the power set of $\{1,\dots,n\}$.


The main aim of this contribution is a study of $n$-ary aggregation functions acting on a bounded distributive lattice $L$ which preserve congruences of $L$. Note that we will show that such aggregation functions are completely characterized by their values at boolean elements  $b\in \{0,1\}^n$. Moreover, our approach shows how idempotent pseudo-boolean functions can be extended into congruence preserving aggregation functions, which surprisingly gives the integration method known as lattice-valued Sugeno integral. When restricting our approach to the real unit interval $[0,1]$, the standard Sugeno integral is recovered, and thus our results bring a new axiomatization of this well-known integral (compare also our recent paper \cite{HMP}).

The paper is organized as follows. In the next section some basic information concerning Sugeno integrals is given. Section \ref{sec:3} brings our main results, characterizing aggregation functions preserving the congruences on $L$. In Section \ref{sec:4} the impact of our new results to the standard Sugeno integral is given. Finally some concluding remarks are added. 

\section{Sugeno integral}\label{sec:2}

Sugeno integral was introduced in 1972 by M. Sugeno in a paper written in Japanese \cite{Sug72}, and it became well-known due to Sugeno's PhD. thesis \cite{Sug74}. For a measurable space $(X,\mathcal{A})$ and a monotone measure $m\colon\mathcal{A}\to [0,1]$, ($m(\emptyset)=0$, $m(X)=1$), the Sugeno integral $\mathsf{Su}_m(f)$ of a measurable function $f\colon X\to [0,1]$ is given by

\begin{equation}\label{eq1}
\mathsf{Su}_m(f)=\bigvee_{t\in[0,1]}\big(t\wedge m(\{x\in X\mid f(x)\geq t\})\big).
\end{equation}

For a finite space $X=\{x_1\dots,x_n\}$, $\mathcal{A}=2^X$, $f\colon X\to[0,1]$ can be identified with a vector $\mathbf{u}\in [0,1]^n$, $\mathbf{u}=(u_1,\dots,u_n)=\big(f(x_1),\dots,f(x_n)\big)$, and formula (\ref{eq1}) can be rewritten into 

\begin{equation}\label{eq2}
\mathsf{Su}_m(f)=\bigvee_{i=1}^{n}\big(u_i\wedge m(\{x\in X\mid f(x)\geq u_i\})\big).
\end{equation}

An alternative formula for the discrete Sugeno integral was proposed in \cite{Mesiar}:

\begin{equation}\label{eq3}
\mathsf{Su}_m(f)=\bigvee_{I\subseteq \{1,\dots,n\}}\big(m(I)\wedge \big(\bigwedge_{i\in I}u_i\big)\big).
\end{equation}

Observe that the Sugeno integral can be seen as a special instance of Ky~Fan metric \cite{Ky Fan} as a distance of the function $f$ and the zero function $\mathbf{0}$.
There are several properties of the discrete Sugeno integral and some of their settings yield an axiomatic characterization of this integral. First of all, for a fixed $m\in\mathbb{N}$, $\mathsf{Su}_m$ can be seen as an aggregation function \cite{Grabisch et al 2009}, i.e., $\mathsf{Su}_m\colon [0,1]^n\to [0,1]$ is non-decreasing in each coordinate, and it satisfies two boundary conditions $\mathsf{Su}_m(\mathbf{0})=0$ and $\mathsf{Su}_m(\mathbf{1})=1$. Next, $\mathsf{Su}_m$ is
\begin{itemize}
\item[--] comonotone maxitive, i.e., $\mathsf{Su}_m(f\vee g)=\mathsf{Su}_m(f)\vee \mathsf{Su}_m(g)$ whenever $f$ and $g$ are comonotone (meaning that they are measurable with respect to a single chain in $2^X$);
\item[--] min-homogeneous, i.e., $\mathsf{Su}_m(\mathbf{c}\wedge f)=c\wedge \mathsf{Su}_m(f)$ for any constant $c\in [0,1]$, $\mathbf{c}=(c,\dots,c)\in[0,1]^n$;
\item[--] horizontally maxitive, i.e., $\mathsf{Su}_m(f)=\mathsf{Su}_m(\mathbf{c}\wedge f)\vee \mathsf{Su}_m(f_c)$ for any $c\in [0,1]$, where $f_c(x_i)=0$ if $f(x_i)\leq c$ and $f_c(x_i)=f(x_i)$ otherwise (observe that $f_c$ is the smallest function on $[0,1]^n$ such that $f=(\mathbf{c}\wedge f)\vee f_c$ );
\item[--] $\mathsf{Su}_m(1_E)=m(E)$, where $1_E$ is the characteristic function of a set $E\subseteq X$;
\item[--] idempotent, i.e., $\mathsf{Su}_m(\mathbf{c})=c$ for any $c\in [0,1]$.
\end{itemize}

For these and several other properties of the discrete Sugeno integral we refer to \cite{Benvenuti,Marichal PhD} and \cite{Couceiro}.
Based on the above references, the Sugeno integral can be characterized as an $[0,1]^n\to [0,1]$ aggregation function which is comonotone maxitive and min-homogeneous. Observe that the comonotone maxitivity can be replaced by the horizontal maxitivity. For some other axiomatizations of the Sugeno integral see \cite{Couceiro}. Marichal \cite{Marichal} has observed an important link between the lattice polynomials on $[0,1]$ and the Sugeno integral. More precisely, he has shown that the class of all Sugeno integrals on $X$ with cardinality $n$ coincides with the class of all polynomial functions $p\colon L^n\to L$, $L=[0,1]$, which are idempotent. This result applies to discrete Sugeno integral defined on any bounded chain $L$, considering the formulae (\ref{eq1})--(\ref{eq3}), and replacing $[0,1]$ by $L$. Also the above mentioned axiomatizations of the Sugeno integral can be extended to any chain $L$. However, in the case of a general bounded distributive lattice $(L,0,1,\leq)$, formulae \eqref{eq1}-\eqref{eq3} are no more equivalent, in general. Following Marichal \cite{Marichal}, we can consider a lattice valued measure $m\colon 2^X\to L$, $m(\emptyset)=0$, $m(X)=1$, $m(E_1)\leq m(E_2)$ whenever $E_1\subseteq E_2\subseteq X$, and for any $f\colon X\to L$ define a discrete $L$-valued Sugeno integral by 

\begin{equation}\label{eq4}
\mathsf{Su}_m(f)=\bigvee_{I\subseteq \{1,\dots,n\}}\big(m(I)\wedge \big(\bigwedge_{i\in I}f(x_i)\big)\big).
\end{equation}

For more details we recommend \cite{Couceiro}.

\section{Compatible aggregation functions on distributive lattices}\label{sec:3}

In this section we clarify the connection between monotone compatible functions on bounded distributive lattices and their weighted lattice polynomials. As the main result we will show that these functions can be identified with Sugeno integrals. 

Recall that a lattice $L$ is distributive, if it satisfies one (or, equivalently, both) of the distributive identities
$$ a\vee(b\wedge c)=(a\vee b)\wedge (a\vee c),\quad a\wedge(b\vee c)=(a\wedge b)\vee (a\wedge c)$$
for all $a,b,c\in L$.
 
\begin{definition}
Let $L$ be a lattice. A binary relation $R\subseteq L^2$ is \textit{compatible} on the lattice $L$ if $(a,b),(c,d)\in R$ imply $(a\vee c,b\vee d)\in R$ and $(a\wedge c,b\wedge d)\in R$ for any $a,b,c,d\in L$. By a \textit{congruence} on $L$ we understand any compatible equivalence on $L$.\end{definition}

In the sequel, for a congruence $\Theta$ on $L$ and $a\in L$, the set $\{x\in L: (a,x)\in\Theta\}$ denotes the congruence class containing the element $a$, and $a\equiv b\pmod\Theta $ will denote the fact that $a$ and $b$ belong to the same congruence class.

\begin{definition}
Let $L$ be a lattice and $n\in \mathbb{N}\cup \{0\}$ be a non-negative integer. By an $n$-ary {\it weighted polynomial}\footnote{We use the name weighted polynomial introduced   in \cite{Marichal}, although in algebraic terminology such functions are called just polynomials.} on the lattice $L$ we mean any function $p: L^n\to L$ defined inductively as follows:
\begin{enumerate}
\item[--] for each $i\in \{1,\dots,n\}$, the $i$-th projection $p(x_1,\dots,x_n)=x_i$ is a weighted polynomial,
\item[--] any constant function $p(x_1,\dots,x_n)=a$ for $a\in L$ is a weighted polynomial,
\item[--] if $p_{1}(x_1,\dots,x_n)$ and $p_{2}(x_1,\dots,x_n)$ are weighted polynomials, then so does the functions $p_{1}(x_1,\dots,x_n)\vee p_{2}(x_1,\dots,x_n)$ and $p_{1}(x_1,\dots,x_n)\wedge p_{2}(x_1,\dots,x_n)$,
\item[--] any weighted polynomial is obtained by finitely many of the preceding steps. 
\end{enumerate}
\end{definition}

Informally, weighted lattice polynomials are functions obtained by composing variables and constant functions by using of lattice operations.

\begin{definition}
Let $L$ be a lattice. A function $f:\, L^n\mapsto L$  is called \textit{compatible} if for any congruence $\Theta$ on $L$, 
$$ f(x_1,\dots,x_n)\equiv f(y_1,\dots,y_n)\pmod\Theta,$$
provided $x_i\equiv y_i\pmod\Theta$ for all $i\in \{1,\dots,n\}$.
\end{definition}

To simplify expressions, for any $n$-ary function $f: L^n\to L$ on a lattice $L$ and $\mathbf{x}=(x_1,\dots,x_n), \mathbf{y}=(y_1,\dots,y_n)\in L^n$, we put $f({\mathbf x}):=f(x_1,\dots,x_n)$, and $\mathbf{x}\leq \mathbf{y}$ iff $x_i\leq y_i$ for all $i\in \{1,\dots,n\}$. 

It can be easily seen that for any lattice, its weighted polynomials are always compatible functions. 
Compatible functions on distributive lattices have been studied in deep by many authors, we refer the reader to \cite{Gr1} or \cite{HP}.

The set $\mathrm{Con}\,L$ of all congruences of $L$ is closed under arbitrary intersections, hence $\mathrm{Con}\,L$ forms a complete lattice with respect to the set inclusion. Consequently, for any two elements $a,b\in L$ there is a least congruence $\Theta_{a,b}$ containing the pair $(a,b)$, called the principal congruence generated by the singleton $(a,b)$.


In order to make the paper self-contained as much as possible, we provide the following important lemma, characterizing the principal congruences on distributive lattices, together with its proof.

\begin{lemma}[\cite{Gr2}, p. 138, Theorem 141]\label{con_lem}
Let $L$ be a distributive lattice, $a,b, x, y\in L$, and let $a\leq b$. 
Then
$$ x\equiv y \pmod{\Theta_{a,b}}\quad \mbox{iff}\quad b\vee x=b\vee y\ \mbox{and}\ a\wedge x=a\wedge y.$$
\end{lemma}

\begin{proof}

Define a binary relation $\Theta\subseteq L\times L$ by $x\equiv y \pmod{\Theta}$, provided 
$$b\vee x=b \vee y\quad \mbox{and}\quad a\wedge x= a\wedge y.$$ 
It can be easily seen that $\Theta$ is reflexive, symmetric and transitive, i.e., it is an equivalence relation on $L$. If $x\equiv y \pmod{\Theta}$ and $u\equiv v \pmod{\Theta}$, then 
$$ b\vee (x\vee u)=(b\vee x)\vee (b\vee u)= (b\vee y)\vee (b\vee v)= b\vee (y\vee v)$$ 
and
$$ a\wedge (x\vee u)=(a\wedge x)\vee (a\wedge u)= (a\wedge y)\vee (a\wedge v)= a\wedge (y\vee v).$$
Hence, $x\vee u\equiv y\vee v \pmod{\Theta}$. Similarly, $x\wedge u\equiv y\wedge v \pmod{\Theta}$ and we conclude that $\Theta$ is a congruence relation on $L$. Moreover, $a\equiv b\pmod{\Theta}$ can be easily verified.

Further, let $\Psi$ be any congruence relation such that $a\equiv b\pmod{\Psi}$. We show that $\Theta\subseteq \Psi$. For this suppose $x\equiv y \pmod{\theta}$. Since
$a\vee x\equiv b\vee x \pmod{\Psi}$, $a\wedge x\equiv b\wedge x \pmod{\Psi}$ and $b\vee x=b\vee y$, $a\wedge x=a\wedge y$, we obtain
\begin{equation*} 
\begin{split}
x=&(a\wedge x)\vee x=(a\wedge y)\vee x =(a\vee x) \wedge (y\vee x)\equiv (b\vee x)\wedge (y\vee x) \\
=&(b\vee y)\wedge (y\vee x) = (b\wedge x)\vee y \equiv (a\wedge x)\vee y =(a\wedge y) \vee y=y. \pmod{\Psi}
\end{split}
\end{equation*}
This shows that $\Theta\subseteq \Psi$ for any congruence $\Psi$ with $a\equiv b \pmod{\Psi}$, i.e., $\Theta=\Theta_{a,b}$. Hence, $b\vee x=b\vee y$ and $a\wedge x=a\wedge y$ implies $x\equiv y\pmod{\Theta_{a,b}}$, while $b\vee x\neq b\vee y$ or $a\wedge x\neq a\wedge y$ yields $x\not\equiv y\pmod{\Theta_{a,b}}$.
\end{proof}

Recall that the median function $\mathsf{med}\colon L^3\to L$ on a lattice $L$ is defined by 
$$ \mathsf{med}(x,y,z)=(x\vee y)\wedge (y\vee z)\wedge (z\vee x).$$

Note that in distributive lattices,  $\mathsf{med}(x,y,z)=(x\wedge y)\vee (y\wedge z)\vee (z\wedge x)$ and if $x\leq z$ then 
$\mathsf{med}(x,y,z)=(x\vee y)\wedge z= x\wedge (y\vee z)$. 

Let $L$ be a bounded distributive lattice with $0$ and $1$ as its bottom and top element, respectively.
Let $n\geq 1$ be a positive integer. For $k\in\{1,\dots,n\}$ and any $\mathbf{x}=(x_1,\dots,x_n)\in L^n$ the elements $\mathbf{x}^0_k,\mathbf{x}^1_k\in L^n$ are defined by  

\begin{equation}\label{nne0} \mathbf{x}^0_k=(x_1,\dots,x_{k-1},0,x_{k+1},\dots x_n), \end{equation}
\begin{equation}\label{nne1} \mathbf{x}^1_k=(x_1,\dots,x_{k-1},1,x_{k+1},\dots x_n). \end{equation}

The following theorem relates compatibility with the median-based decomposition property.

\begin{theorem}
Let $f\colon L^n\to L$ be a nondecreasing function. Then $f$ is compatible if and only if for each $k\in\{1,\dots,n\}$
\begin{equation}\label{eq_m1}
 f(\mathbf{x})=\mathsf{med}\big(f(\mathbf{x}^0_k),x_k,f(\mathbf{x}^1_k)\big),\ \mbox{for all}\ \mathbf{x}\in L^n.
\end{equation}
\end{theorem}

\begin{proof}

First we show that the condition \eqref{eq_m1} is valid for all unary compatible functions, i.e., given a nondecreasing compatible function $f\colon L \to L$ the equality 
$$f(x)=(f(0)\vee x)\wedge f(1)=\mathsf{med}\big(f(0),x,f(1)\big)$$ holds for all $x\in L$.

Since $f$ is compatible, from $0\equiv a\pmod{\Theta_{0,a}}$ we have $f(a)\equiv f(0)\pmod{\Theta_{0,a}} $,
while $a\equiv 1\pmod{\Theta_{a,1}}$ yields $ f(a)\equiv f(1)\pmod{\Theta_{a,1}}$.
According to Lemma \ref{con_lem} we have  
$$ a\vee f(a)= a\vee f(0)\quad \mbox{and} \quad a\wedge f(a)= a\wedge f(1). $$
From this, using distributivity of $L$ and the fact that $f$ is nondecreasing, we obtain 

\begin{equation*}
\begin{split}
f(a) &=f(a)\vee \big(a\wedge f(a)\big)=f(a)\vee \big(a\wedge f(1)\big)=\big(f(a)\vee a\big)\wedge \big(f(a)\vee f(1)\big)\\
     &=\big(f(0)\vee a\big)\wedge f(1)=\mathsf{med}\big(f(0),a,f(1)\big).
\end{split}
\end{equation*}

Further, let $n\geq 2$ and $f\colon L^n\to L$ be a nondecreasing compatible function. For $k\in\{1,\dots,n\}$ and for an arbitrary $(n-1)$-tuple $\overline{\mathbf{a}}_k=(a_1,\dots, a_{k-1},a_{k+1},\dots,a_n)\in L^{n-1}$ we define the unary function $f_{\overline{\mathbf{a}}_k}(x)\colon L\to L$ given by
$$ f_{\overline{\mathbf{a}}_k}(x)=f(a_1,\dots,a_{k-1},x,a_{k+1},\dots,a_n).$$ 
Obviously, $f_{\overline{\mathbf{a}}_k}$ is nondecreasing as well as compatible. Hence for any $a\in L$ we obtain 
$$f_{\overline{\mathbf{a}}_k}(a)=\big(f_{\overline{\mathbf{a}}_k}(0)\vee a\big)\wedge f_{\overline{\mathbf{a}}_k}(1)=\mathsf{med}\big(f_{\overline{\mathbf{a}}_k}(0),a,f_{\overline{\mathbf{a}}_k}(1)\big).$$

Moreover for $\mathbf{a}=(a_1,\dots,a_{k-1},a,a_{k+1},\dots,a_n)$ we have $f_{\overline{\mathbf{a}}_k}(0)=f(\mathbf{a}^0_k)$ and $f_{\overline{\mathbf{a}}_k}(1)=f(\mathbf{a}^1_k)$, where $\mathbf{a}^0_k$ and $\mathbf{a}^1_k$ are defined by \eqref{nne0} and \eqref{nne1} respectively.
Since $\overline{\mathbf{a}}_k$ and $a$ were arbitrary, it follows that $\eqref{eq_m1}$ holds.
\medskip

Conversely, assume that $f\colon L^n \to L$, $n\geq 1$ satisfies \eqref{eq_m1}. Then for each $k\in\{1,\dots,n\}$ and all $\overline{\mathbf{a}}_k\in L^{n-1}$ the unary function $f_{\overline{\mathbf{a}}_k}$ is a polynomial, i.e., it is compatible. Let $\Theta$ be a congruence relation on $L$, and $\mathbf{c}=(c_1,\dots,c_n),\mathbf{d}=(d_1,\dots,d_n)\in L^n$ be such that $c_i\equiv d_i\pmod{\Theta}$ for all $i\in\{1,\dots,n\}$. Using compatibility of the unary functions we obtain

\begin{equation*}
\begin{split}
f(c_1,c_2,\dots,c_{n-1},c_n)&\equiv f(d_1,c_2,\dots,c_{n-1},c_n) \pmod{\Theta}\\
f(d_1,c_2,\dots,c_{n-1},c_n)&\equiv f(d_1,d_2,\dots,c_{n-1},c_n) \pmod{\Theta}\\
                    &\hspace{0.20cm}\vdots                    \\
f(d_1,d_2,\dots,d_{n-1},c_n)&\equiv f(d_1,d_2,\dots,d_{n-1},d_n) \pmod{\Theta},\\
\end{split}
\end{equation*}
and the transitivity of $\Theta$ yields $f(\mathbf{c})\equiv f(\mathbf{d})\pmod{\Theta}$. This shows that $f$ is compatible.
\end{proof}

\begin{theorem}
Let $f_1\colon L^n\to L$ and $f_2\colon L^n\to L$ be two nondecreasing compatible functions. If $f_1(\mathbf{b})=f_2(\mathbf{b})$ for all $\mathbf{b}\in\{0,1\}^n$, then $f_1(\mathbf{x})=f_2(\mathbf{x})$ for all $\mathbf{x}\in L^n$.
\end{theorem}

\begin{proof}
We proceed by induction with respect to the arity of the functions. Note that by the previous theorem any nondecreasing compatible function satisfies the median-based decomposition property \eqref{eq_m1}. Obviously, for $n=1$ we obtain 
$$ f_1(x)=\mathsf{med}\big(f_1(0),x,f_1(1)\big)=\mathsf{med}\big(f_2(0),x,f_2(1)\big)=f_2(x).$$ 
Further, assume that the assertion is valid for some $n\geq 1$. Let $f_1,f_2\colon L^{n+1}\to L$ satisfy the assumptions of the theorem. For $i\in\{1,2\}$ define $\overline{f_i^0}\colon L^n\to L$ by 
$$\overline{f_i^0}(x_1,\dots x_n)=f_i(x_1,\dots,x_n,0)$$ 
and similarly $\overline{f_i^1}\colon L^n\to L$ by 
$$\overline{f_i^1}(x_1,\dots x_n)=f_i(x_1,\dots,x_n,1).$$

Since $f_1(\mathbf{b})=f_2(\mathbf{b})$ for all $\mathbf{b}\in \{0,1\}^{n+1}$, it is easily seen that $\overline{f_1^0}(\mathbf{c})=\overline{f_2^0}(\mathbf{c})$ and $\overline{f_1^1}(\mathbf{c})=\overline{f_2^1}(\mathbf{c})$ for all $\mathbf{c}\in \{0,1\}^n$. Moreover these functions fulfill also condition \eqref{eq_m1}, hence according to the induction hypothesis for $\mathbf{x}=(x_1,\dots,x_n,x_{n+1})\in L^{n+1}$ and the corresponding first $n$ coordinates $\overline{\mathbf{x}}=(x_1,\dots,x_n)\in L^n$ we obtain
\begin{equation*}
\begin{split}
 f_1(\mathbf{x})&=\mathsf{med}\big(f_1(\mathbf{x}^0_{n+1}),x_{n+1},f_1(\mathbf{x}^1_{n+1})\big)=\mathsf{med}\big(\overline{f_1^0}(\overline{\mathbf{x}}),x_{n+1},
 \overline{f_1^1}(\overline{\mathbf{x}})\big) \\
 &=\mathsf{med}\big(\overline{f_2^0}(\overline{\mathbf{x}}),x_{n+1}, \overline{f_2^1}(\overline{\mathbf{x}})\big)=\mathsf{med}\big(f_2(\mathbf{x}^0_{n+1}),x_{n+1},f_2(\mathbf{x}^1_{n+1})\big)=f_2(\mathbf{x}).
\end{split}
\end{equation*} 
\end{proof}

\begin{corollary}\label{cor_m1}
Any nondecreasing compatible function $f\colon L^n\to L$ is uniquely determined by its values at boolean elements $\mathbf{b}\in \{0,1\}^n$.
\end{corollary}

Let us note that Corollary \ref{cor_m1} is due to G. Gr\"atzer \cite{Gr1}, but his proof is completely different.  
It uses the fact that every distributive lattice $L$ can be (canonically) embedded into a Boolean algebra (using a set-theoretical representation of $L$) \cite{Gr2}.

Characterization of  functions satisfying \eqref{eq_m1} as precisely those functions which are weighted lattice polynomials was established by Marichal in \cite{Marichal}. 
However, Corollary \ref{cor_m1} enables to associate a unique $L$-valued monotone measure to each monotone compatible function. Using this fact, we present a different proof based on the modification of the approach developed in \cite{HP}.
  
Given a  nondecreasing compatible function $g\colon L^n\to L$, for any $\mathbf{b}\in \{0,1\}^n$ consider the functions 
\begin{equation}\label{eq5}
G_{\mathbf{b}}(\mathbf{x}):=g(\mathbf{b})\wedge \bigwedge \big\{x_i\mid\, i\in \mathbf{b}^{-1}(1)\big\},
\end{equation}
where $\mathbf{b}^{-1}(1)=\{i\in \{1,\dots,n\}: b_i=1\}$.

\begin{theorem}
For any monotone compatible function $g$ the following equality holds:
\begin{equation}\label{eq6}
g(\mathbf{x})=\bigvee \big\{G_{\mathbf{b}}(\mathbf{x})\mid\, \mathbf{b}\in \{0,1\}^n\big\}.
\end{equation}
\end{theorem}

\begin{proof}
Since the functions on both sides are compatible,
due to Corollary \ref{cor_m1}, it is enough to prove the above equality only for boolean inputs $\mathbf{x}\in\{0,1\}^n$. Consider an arbitrary $\mathbf{b}\in \{0,1\}^n$.
We have the following possibilities:
\begin{itemize}
\item[(1)] Let $\mathbf{b}\nleq \mathbf{x}$. Then there is $j\in\{1,\dots,n\}$ with $b_j=1$ (i.e., $j\in \mathbf{b}^{-1}(1)$) and $x_j=0$, which yields $\bigwedge \{x_i\mid i\in \mathbf{b}^{-1}(1)\}=0$. Consequently, we obtain $G_{\mathbf{b}}(\mathbf{x})=g(\mathbf{b})\wedge 0=0$.

\item[(2)] Let $\mathbf{b}=\mathbf{x}$. Then, evidently, $\bigwedge \{x_i\mid i\in \mathbf{b}^{-1}(1)\}=\bigwedge \{x_i\mid i\in \mathbf{x}^{-1}(1)\}=\bigwedge 1=1$ whenever $\mathbf{b}^{-1}(1)\neq\emptyset$, and it equals $\bigwedge \emptyset=1$ in case $\mathbf{b}^{-1}(1)=\emptyset$. In both cases we obtain 
$G_{\mathbf{b}}(\mathbf{x})=g(\mathbf{b})\wedge 1=g(\mathbf{b})=g(\mathbf{x})$ since we assumed $\mathbf{b}=\mathbf{x}$.
 
\item[(3)] Assume $\mathbf{b}<\mathbf{x}$. Then as the function $g$ is monotone and $\mathbf{b}<\mathbf{x}$, we conclude $G_{\mathbf{b}}(\mathbf{x})\leq g(\mathbf{b})\leq g(\mathbf{x})$.
\end{itemize}
The above discussion leads to the desired equality $$g(\mathbf{x})=\bigvee \big\{G_{\mathbf{b}}(\mathbf{x})\mid \mathbf{b}\in \{0,1\}^n\big\}.$$  
\end{proof}

Consequently, we obtain the following conclusion.

\begin{corollary} 
Compatible aggregation functions on distributive lattices are just their weighted idempotent lattice
polynomials.
\end{corollary}

Let us stress that this statement does not depend on the cardinality of a lattice $L$, and hence it holds also in a classical case when $L=[a,b]$ is any bounded interval of reals.

\section{Sugeno integral as a compatible aggregation function}\label{sec:4}

Consider a lattice $([0,1],0,1,\leq)$, where $L = [0,1]$ is the real unit interval.
Then each element $\mathbf{a}\in\{0,1\}^n$ can be identified with a characteristic function of
a subset $I$ of $\{1,\dots,n\}$, $\mathbf{a} = 1_I$. Comparing formulae (\ref{eq3}) and (\ref{eq6}), the next
result is obtained easily.

\begin{theorem}\label{thm2}
 Let $L$ be a bounded distributive lattice and $A\colon L^n \to L$ be an aggregation function. Then the
following are equivalent:
\begin{itemize}
\item[(i)] A is a compatible function
\item[(ii)] there is a monotone $L$-valued measure $m$ on $\{1,\dots,n\}$ so that $A = \mathsf{Su}_m$, i.e., $A$ is
the Sugeno integral with respect to the measure $m$.
\end{itemize}
\end{theorem}

Note that the monotone measure $m$ in Theorem \ref{thm2} is given by $m(I) = A(1_I)$. Our
result brings a new characterization of the classical Sugeno integral in
discrete setting. Evidently, due to (\ref{eq4}), Theorem \ref{thm2} can be extended to any
bounded distributive lattice $(L, 0, 1, \leq)$.

Based on our results, the following consequences for particular types of lattices can be shown straightforwardly:

\begin{itemize}
\item[-] if the considered bounded distributive lattice $L$ is a direct product of bounded distributive lattices $L_k,\, k\in K$ \cite{Gr2} , then the $n$-ary Sugeno integral on $L$ with respect to an $L$-valued measure $m$ can be represented in the form of a direct product of $n$-ary Sugeno integrals on $L_k$ with respect to measures $m_k$, where $m_k$ is the $k$-th projection of $m$ into $L_k$

\item[-] if the considered bounded distributive lattice $L$ is a horizontal sum \cite{Gr2} of bounded distributive lattices $L_k, \,k\in K$, then the $n$-ary Sugeno integral on $L$ with respect to an $L$-valued measure $m$ (with integrand $u$) can be seen as supremum of $n$-ary Sugeno integrals on $L_k$ with respect to measures $m_k$ and with integrand $u^k$, where $m_k(I) = m(I)$ if $m(I)$ is from $L_k$, and $m_k(I) = 0$ otherwise, and similarly, for the single components of the integrated vector $u^k$ we have $u^k_i = u_i$  if $u_i$ is from $L_k$, and it is $0$ otherwise ($i = 1,\dots,n$).

\end{itemize}

\section{Concluding remarks}
We have introduced a new property which characterizes the discrete Sugeno
integral not only in its original form, when $[0,1]$-valued functions and
measures are considered, but also in the case of general bounded lattices.
This property, compatibility, has an important impact for decision procedures
which will be the topic of our next investigations. Here we recall only the
next fact: in multicriteria decision problems based on $n$ criteria and dealing
with alternatives described by score vectors from $[0,1]^n$, often the exact
numerical scores are replaced by some ordinal scale, e.g. by a linguistic
scale. The transition from numerical inputs to linguistic values is done by
means of interval partitions of the original scale $[0,1]$. When looking for
normed utility functions (i.e., aggregation functions) where the output
recommendation based on linguistic values does not depend on the numerical
values of score vectors, then due to Theorem \ref{thm2} (applied for $L=[0,1]$), 
only Sugeno integrals (i.e.,
idempotent polynomials) can be considered.

\section*{Acknowledgements}
The first author was supported by the international project Austrian Science Fund (FWF)-Grant Agency of the Czech Republic (GA\v{C}R) number I 1923-N25; the second author by the Slovak VEGA Grant 1/0420/15 and by the NPUII project LQ1602; the third author by the project of Palack\'y University Olomouc IGA PrF2015010 and by the Slovak VEGA Grant no. 2/0044/16.




\end{document}